\documentclass[11pt]{article}
\usepackage[dvipsnames,x11names]{xcolor} 
\usepackage{epsfig}
\usepackage{lscape}
\usepackage{amssymb} 
\usepackage{graphicx}
\usepackage{footnote}
\usepackage[utf8]{inputenc} 
\usepackage{mathtools}
\usepackage{amsthm}
\usepackage{wasysym}
\usepackage{amsfonts}
\usepackage[english]{babel}
\usepackage{bigints}
\usepackage{mathrsfs}
\usepackage{lineno}
\usepackage{nicefrac}
\usepackage{enumerate}
\usepackage{appendix}
\usepackage{graphicx}
\usepackage{commath}
\usepackage{multirow}
\usepackage{tabu}

\usepackage{times}
\usepackage[T1]{fontenc}

\usepackage{algorithm} 
\usepackage{algpseudocode} 
\usepackage{booktabs}

\usepackage[most]{tcolorbox}
\usepackage{mdframed}

\usepackage{amsmath}
\usepackage{mathtools} 
\usepackage{mathrsfs}
\usepackage{tabularx}

\usepackage[]{natbib}
\usepackage{csquotes}

\usepackage[colorlinks=true,citecolor=blue,linkcolor=black,urlcolor=blue]{hyperref}

\usepackage[margin=2cm]{geometry}

\allowdisplaybreaks[1]
\usepackage{pifont}

\makeindex

\usepackage{sectsty}
\sectionfont{\color{BrickRed}}

\sectionfont{\color{RoyalBlue4}}
\subsectionfont{\color{RoyalBlue4}}
\subsubsectionfont{\color{RoyalBlue4}}
\paragraphfont{\color{RoyalBlue4}}

\newtheorem{definition}{Definition}[section]
\newtheorem{theorem}{Theorem}[section]

\newtheorem{remark}[theorem]{Remark}

\newcommand{\imag}{\textbf{i}}

\newcommand{\ImagPart}{\mathfrak{Im}}
\renewcommand{\Im}{\mathfrak{Im}}

\newcommand{\error}{\mathcal{O}}

\begin{document}
	
	\begin{flushleft}
		
		{\huge\bfseries\color{DodgerBlue4} From complex--step differentiation to a general reconstruction framework\par}
		\vspace{1.5em}  

		Rafael Abreu\textsuperscript{1,*} and Chahana Nagesh\textsuperscript{1}\\[0.5em]  
		
		\scriptsize 
		\textsuperscript{1} \textit{Institut de Physique du Globe de Paris, CNRS, Universit\'e de Paris, Paris, France}  \\[1em]  
		\textsuperscript{*} {email: rabreu@ipgp.fr}
		
		\normalsize 
	\end{flushleft}

	\begin{abstract}
		The complex-step method is traditionally derived from the Taylor
		expansion of an analytic function and is widely used as a numerical
		technique for derivative approximation. We present an alternative
		formulation based on the Cauchy--Riemann equations and show that the
		classical complex-step relation arises naturally from the harmonic
		structure of holomorphic functions. In particular, the complex-step
		method admits two complementary harmonic interpretations: as a Cauchy
		problem, in which the derivative is identified with the normal datum of
		the imaginary component on the real axis, and as a reconstruction
		problem in a strip, in which the finite imaginary perturbation provides
		the upper-boundary data. The latter formulation leads explicitly to the
		strip Poisson and conjugate Poisson kernels and their derivatives.
		
		A related harmonic reconstruction framework in the upper half-plane
		leads to the Poisson, conjugate Poisson, and Cauchy kernels as elementary
		reconstruction operators for harmonic and holomorphic functions.
		Extending this reconstruction from ordinary boundary functions to finite
		measures yields the classical Stieltjes transform and its inversion
		formula. The same measure-theoretic structure appears in spectral
		theory, where scalar matrix elements of the resolvent are Stieltjes
		transforms of the associated spectral measures. These results establish
		a common complex-analytic structure connecting complex-step
		differentiation, harmonic reconstruction, Stieltjes inversion, and
		spectral reconstruction, while distinguishing the boundary-value
		problems through which the corresponding information is recovered.
	\end{abstract}
	
	
	\newpage
	
	\section{Introduction}
	
	The complex-step method (CSM), originally introduced by \citet{Squire1998}, is commonly presented as a numerical technique for derivative approximation. Its practical success is generally attributed to the fact that imaginary perturbations avoid the subtractive cancellation errors that limit conventional finite-difference formulas \citep{Squire1998,Abreu201384}. From this viewpoint, the complex perturbation is regarded primarily as a numerical tool for computing derivatives of analytic functions.
		
	A similar use of imaginary perturbations appears in a seemingly unrelated area of mathematics. The classical Stieltjes inversion formula reconstructs a measure from its Stieltjes transform by evaluating its analytic continuation at points \(z=x+\imag\epsilon\) lying in the upper half-plane, where \(x\in\mathbb{R}\) denotes the location on the real axis and \(\epsilon>0\) the distance from the boundary. As the boundary is approached \((\epsilon\rightarrow0^+)\), the imaginary component converges to the underlying measure.
	
	The appearance of the same imaginary perturbation in apparently
	different mathematical constructions suggests the existence of a common
	underlying complex-analytic structure. Although the complex-step method
	and the Stieltjes inversion formula have traditionally been developed
	for different purposes---derivative approximation and measure
	reconstruction, respectively---both make use of analytic continuation
	away from the real axis and recover information through its real and
	imaginary components.
	
	The purpose of the present work is to investigate these connections from
	the viewpoint of harmonic analysis. To do so, we first show that the
	classical complex-step approximation follows directly from the
	Cauchy--Riemann equations, rather than being viewed solely as a
	consequence of the Taylor series expansion. Since the Cauchy--Riemann
	equations imply that the real and imaginary components of a holomorphic
	function are harmonic, the complex-step method admits two complementary
	harmonic interpretations. It can be formulated as a Cauchy problem on
	the real boundary, where the derivative is identified with the normal
	datum of the imaginary component, or as a Dirichlet reconstruction
	problem in a strip, where the finite imaginary perturbation provides the
	data at the upper boundary. A related reconstruction framework arises
	in the upper half-plane through the Poisson and conjugate Poisson
	kernels and extends naturally from ordinary boundary functions to finite
	measures and spectral measures.
	
	The remainder of the paper is organized as follows. We first review the
	classical derivation of the complex-step approximation based on the
	Taylor series expansion and then derive the same relation directly from
	the Cauchy--Riemann equations. We introduce harmonic reconstruction in
	the upper half-plane through the classical Poisson, conjugate Poisson,
	and Cauchy kernels and subsequently develop the two complementary
	harmonic formulations of the complex-step method: reconstruction in a
	strip and reconstruction from Cauchy data. Extending the upper-half-plane
	reconstruction from ordinary boundary functions to finite measures leads
	to the Stieltjes transform and its inversion formula, and the resulting
	measure-theoretic structure is then connected with spectral measures
	through the resolvent.
	
	\section{The complex step method from Taylor series}
	
	The numerical computation of derivatives plays a central role in scientific computing. Unlike conventional finite-difference methods, the complex-step method evaluates the function at an imaginary perturbation \citep{Squire1998,Abreu201384}. To evaluate the function at imaginary perturbations, we first extend the original real-valued function into the complex plane. Let
	\begin{align*}
		F:\Omega\subset\mathbb C\rightarrow\mathbb C
	\end{align*}
	be a holomorphic extension of the real-valued function
\begin{align*}
	f:\Omega\cap\mathbb R\rightarrow\mathbb R,
\end{align*}
	so that
\begin{align*}
	F(x)=f(x),
	\qquad
	x\in\Omega\cap\mathbb R.
\end{align*}

	The standard forward and centered finite difference approximations of the first-order derivative are given by the following expressions respectively \citep{Thomas1995}
	\begin{align}
		& \partial_x f(x) = \frac{f(x+h) - f(x)}{h} + \error \left(h\right), & \partial_x f (x) = \frac{f(x+h) - f(x-h)}{2h} + \error \left(h^2\right).&
		\label{eq.FD_approximations}
	\end{align}
	
	The classical derivation of the complex-step approximation \citep{Squire1998} is obtained by expanding the analytic continuation \(F\) in a Taylor series about the point \(x\),
	\begin{align}
		F(x + \imag h) & = F(x) + \imag h \partial_x F (x) - \frac{h^2}{2} \partial_x^2 F (x)  + ... + \frac{(\imag h)^n}{n!} \partial_x^n F(x)  = \sum_{n=0}^{\infty}\frac{(\imag h)^{n}}{n!}\partial_x^n  F(x),
		\label{eq.Taylor_Series_Deltax}
	\end{align}
	where $\imag$ is the imaginary unit and $h\to 0$. Since \(F(x)=f(x)\in\mathbb R\) on the real axis, taking the imaginary part of eq. \eqref{eq.Taylor_Series_Deltax} yields
	\begin{align}
		\partial_x F (x)  = \partial_x f (x) = \frac{\ImagPart \left[F(x+\imag h)\right]}{h} ,
		\label{eq.GCS_CSFirstOrder}
	\end{align}
	where we have committed the error term corresponding to the truncation of the Taylor series. 
	
	The complex step CS derivative (eq. \eqref{eq.GCS_CSFirstOrder}) is more efficient than conventional FD techniques (eq. \eqref{eq.FD_approximations}) in many applications because it helps to avoid numerical cancellation errors \citep[e.g.][]{anderson2001sensitivity,al2010complex,Voorhees20111146,martins2001connection,martins2003complex,dziatkiewicz2016complex,Jin2010,Kim2006177,Burg2003,Wang2006,ridout2009,Abreu201384,Abreu01072015,ABREU2018390}.
	
	This standard derivation presents the complex-step approximation as a consequence of the Taylor expansion. In the following section, we show that the same approximation follows directly from the Cauchy--Riemann equations. This alternative derivation shows that the complex-step approximation is fundamentally a consequence of the harmonic structure of holomorphic functions, from which the Poisson, Hilbert and Cauchy kernels emerge naturally.
	
	\section{The complex step method from the Cauchy--Riemann equations}
	
	We next show that the complex-step approximation (eq. \eqref{eq.GCS_CSFirstOrder}) is a direct consequence of the Cauchy--Riemann equations.
	
	\begin{theorem}[Complex-step approximation from the Cauchy--Riemann equations]
		Let
		\[
		F:\Omega\subset\mathbb{C}\rightarrow\mathbb{C}
		\]
		be holomorphic in a neighborhood of an interval of the real axis, and assume that
		\[
		F(x)\in\mathbb{R},
		\qquad x\in\Omega\cap\mathbb{R}.
		\]
		Then, for every
		\(x\in\Omega\cap\mathbb{R}\),
		\begin{align}
			\partial_xF(x)
			=
			\lim_{h\rightarrow0}
			\frac{\ImagPart\!\left[F(x+\imag h)\right]}{h}.
			\label{eq.Complex_Step_relation}
		\end{align}
	\end{theorem}
	
	\begin{proof}
	Since \(F\) is real-valued on the real axis,
\begin{align*}
	V(x,0)=0.
\end{align*}
From the Cauchy--Riemann equations we can write,
\begin{align*}
	\partial_x F(x)
	=
	\partial_x U(x,0)
	=
	\partial_y V(x,0).
\end{align*}
By the definition of the partial derivative,
\begin{align*}
	\partial_y V(x,0)
	&=
	\lim_{h\rightarrow0}
	\frac{V(x,h)-V(x,0)}{h} \\
	&=
	\lim_{h\rightarrow0}
	\frac{V(x,h)}{h}.
\end{align*}
Since
\begin{align*}
	V(x,h)=\ImagPart\!\left[F(x+\imag h)\right],
\end{align*}
we obtain
\begin{align*}
	\partial_x F(x)
	=
	\lim_{h\rightarrow0}
	\frac{\ImagPart\!\left[F(x+\imag h)\right]}{h},
\end{align*}
which proves the theorem.
	\end{proof}

\section{The complex step method as a reconstruction problem}

The previous theorem establishes that the imaginary component of the
analytic continuation is a harmonic function. A fundamental property
of harmonic functions is that their values within the domain are
determined by appropriate boundary data. This means that the complex-step
method can be interpreted not merely as a differentiation technique,
but within the framework of harmonic reconstruction. This perspective
naturally places the complex-step method within the classical theory of
harmonic boundary-value problems, which we now develop.

Since \(F\) is holomorphic we can write
\begin{align*}
	F(z) = U(x,y) + \imag V(x,y),
\end{align*}
and the real \(U(x,y)\) and imaginary \(V(x,y)\) components satisfy the
Cauchy--Riemann equations and, consequently, both components are
harmonic, i.e.,
\begin{align*}
	\Delta U = 0, \qquad \Delta V = 0,
\end{align*}
where \(\Delta=\partial_x^2+\partial_y^2\) is the Laplacian. We first develop the classical reconstruction of harmonic and holomorphic functions from boundary data in the upper half-plane. We then show that the boundary conditions associated with the complex-step method lead naturally to the corresponding reconstruction problem in a strip.

	\subsection{Reconstruction from a Dirac boundary}
	
    Since the imaginary component $V(x,y)$ can be recovered, up to an additive constant, as the harmonic conjugate of the real component $U(x,y)$ using the Cauchy--Riemann equations, it is sufficient to reconstruct only $U$. We first consider the reconstruction problem for a Dirac delta prescribed at the boundary. Given boundary data
	\begin{align*}
		U(x,0)=u_0(x) ,
	\end{align*}
	we seek a harmonic function satisfying
	\begin{align*}
		\Delta U=0,\qquad y>0,\qquad \lim_{y\to0^+}U(x,y)=u_0(x).
	\end{align*}
	The solution of the Dirichlet problem in the upper half-plane is obtained by prescribing the elementary boundary datum. We first consider the elementary boundary datum as a Dirac delta
	\begin{align*}
		u_0(x)=\delta(x),
	\end{align*}
	for which the solution of the Laplace equation is given by the Poisson kernel
	\begin{align}
		P_y(x)=\frac{1}{\pi}\frac{y}{x^2+y^2},
		\label{eq.Poisson_kernel}
	\end{align}
	which satisfies
	\begin{align*}
		\Delta P_y=0,\qquad y>0,\qquad \lim_{y\to0^+}P_y(x)=\delta(x).
	\end{align*}
	
	The Poisson kernel therefore reconstructs only the real component of the holomorphic function. To recover the complete analytic continuation, the corresponding iaginary component must be determined. Since the real and imaginary components are harmonic conjugates, the Cauchy--Riemann equations uniquely determine $V$ from the reconstructed field $U$. For the elementary solution
	\begin{align}
		U(x,y)=P_y(x)=\frac{1}{\pi}\frac{y}{x^2+y^2},
	\end{align}
	we compute
	\begin{align}
		\partial_xU(x,y)
		=
		-\frac{2xy}{\pi(x^2+y^2)^2}.
	\end{align}
	Since the Cauchy--Riemann equations require
	\begin{align*}
		\partial_xU=\partial_yV,
	\end{align*}
	we seek a function \(V\) satisfying this relation. Integrating with respect to \(y\) gives
	\begin{align}
		V(x,y)
		&=
		\int
		-\frac{2xy}{\pi(x^2+y^2)^2}\,dy =
		\frac{1}{\pi}\frac{x}{x^2+y^2}+C(x).
	\end{align}
	where \(C(x)\) is an arbitrary function of \(x\). The second Cauchy--Riemann equation,
	\begin{align*}
		\partial_yU=-\partial_xV,
	\end{align*}
	implies that \(C'(x)=0\), so \(C(x)\) is constant. Thus, up to an additive constant,
	\begin{align}
		V(x,y)
		=
		\frac{1}{\pi}\frac{x}{x^2+y^2} = Q_y (x),
	\end{align}
	where \(Q_y\) is the conjugate Poisson kernel. The elementary
	holomorphic solution is therefore
	\begin{align}
\begin{aligned}
			F(z)
	&=
	\underbrace{P_y(x)}_{\text{Poisson kernel}}
	+
	\imag\,
	\underbrace{Q_y(x)}_{\text{conjugate Poisson kernel}} \\
	&=
	\frac{1}{\pi}
	\frac{y+\imag x}{x^2+y^2}
	=
	\frac{\imag}{\pi z},
	\qquad z=x+\imag y.
	\label{eq.Cauchy_Hilbert_Poisson}
\end{aligned}
	\end{align}

	\subsection{Reconstruction from arbitrary boundary data}
	
	Since the Laplace equation is linear, arbitrary boundary data can now be reconstructed by superposing translated copies of the elementary solution \citep{axler2001harmonic}. For general boundary data $u_0(x)$, the harmonic solution is obtained by convolving the Poisson kernel $P_y$ (eq. \eqref{eq.Poisson_kernel}) with the prescribed boundary data  
		\begin{align}
			U(x,y)
			=
			\int_{-\infty}^{\infty}
			P_y(x-t)\,u_0(t)\,dt
			=
			(P_y*u_0)(x).
			\label{eq.Convolution_real}
		\end{align}
	The corresponding imaginary component, recovered as the harmonic conjugate of $U$, is 
	\begin{align}
		V(x,y)
		=
		\int_{-\infty}^{\infty}
		Q_y(x-t)\,u_0(t)\,dt
		=
		(Q_y*u_0)(x),
	\end{align}
	where $	Q_y(x)$ is the conjugate Poisson kernel. Note that in the limit \(y\to0^+\), the conjugate Poisson kernel
	\(Q_y(x)\) is singular at \(x=0\). Thus, the limit is understood
	in the sense of the Cauchy principal value,
	\begin{align}
		\lim_{y\to0^+}Q_y(x)
		=
		\frac{1}{\pi}\,
		\mathrm{p.v.}\!\left(\frac{1}{x}\right) ,
	\end{align}
	where
	\begin{align*}
		\mathrm{p.v.}\!\int_{-\infty}^{\infty}\frac{f(t)}{x-t}\,dt
		:=
		\lim_{\epsilon\to0^+}
		\left(
		\int_{-\infty}^{x-\epsilon}
		\frac{f(t)}{x-t}\,dt
		+
		\int_{x+\epsilon}^{\infty}
		\frac{f(t)}{x-t}\,dt
		\right).
	\end{align*}
	Consequently, while the prescribed boundary value of the real component is $U(x,0)=u_0(x)$, the boundary value of its harmonic conjugate is determined by the corresponding principal-value convolution,
	\begin{align}
		V(x,0)
		=
		\frac{1}{\pi}\,
		\mathrm{p.v.}\!\int_{-\infty}^{\infty}
		\frac{u_0(t)}{x-t}\,dt.
	\end{align}
	Thus, the boundary value of \(V\) is not prescribed independently,
	but follows from the harmonic conjugacy relation.
	
	While the Poisson kernel converges to the Dirac distribution and reconstructs the prescribed boundary data, the conjugate Poisson kernel converges to the singular kernel defining the classical Hilbert transform \citep{axler2001harmonic}. Combining the Poisson reconstruction with its harmonic conjugate recovers
	the complete holomorphic function,
	\begin{align}
		F(z)
		=
		\underbrace{U(x,y)}_{\text{Poisson}}
		+
		\imag\,
		\underbrace{V(x,y)}_{\text{conjugate Poisson}},
		\qquad z=x+\imag y.
		\label{eq.Cauchy_Hilbert_Poisson_General_boundary}
	\end{align}
	Using convolution notation, this becomes
	\begin{align}
		F(z)
		=
		(P_y*u_0)(x)
		+
		\imag\,(Q_y*u_0)(x),
		\qquad z=x+\imag y.
		\label{eq.Integral_Cauchy_Hilbert_Poisson}
	\end{align}

\subsection{The complex step method from a strip}

The preceding reconstruction was formulated in the upper half-plane from prescribed boundary data for one harmonic component. The
complex-step method has a different boundary structure. Since \(F\) is real on the real axis, its imaginary component satisfies
\begin{align*}
	V(x,0)=0,
\end{align*}
whereas evaluation at the complex step \((x+\imag h)\) determines its value at the parallel boundary
\begin{align*}
	V(x,h)
	=
	\ImagPart\!\left[F(x+\imag h)\right].
\end{align*}
The natural reconstruction domain is therefore the strip \(0<y<h\). We now show that the complex-step equation (eq. \eqref{eq.Complex_Step_relation}) arises directly from this strip reconstruction.

\begin{theorem}[Complex-step differentiation from harmonic reconstruction in a strip]
	Let
	\begin{align*}
		F(z)=U(x,y)+\imag V(x,y)
	\end{align*}
	be holomorphic in a neighborhood of the closed strip \(0\leq y\leq h\), with real values on the real axis, so that
	\begin{align*}
		V(x,0)=0.
	\end{align*}
	For each \(h>0\), consider the strip \(0<y<h\) and define the
	upper-boundary data
	\begin{align}
		g_h(x)
		=
		V(x,h)
		=
		\ImagPart\!\left[F(x+\imag h)\right].
		\label{eq.CS_strip_data}
	\end{align}
	The harmonic reconstruction of \(V\) in the strip is given by
	\begin{align}
		V(x,y)
		=
		\left(P_h(\cdot,y)*g_h\right)(x),
		\qquad 0<y<h,
		\label{eq.CS_strip_convolution}
	\end{align}
	where
	\begin{align}
		P_h(x,y)
		=
		\frac{1}{2h}
		\frac{\sin(\pi y/h)}
		{\cosh(\pi x/h)+\cos(\pi y/h)} ,
		\label{eq.CS_strip_Poisson}
	\end{align}
	is the Poisson kernel associated with the upper boundary of the strip. Then
	\begin{align}
\begin{aligned}
			\partial_xF(x)
	&=
	\lim_{h\to0^+}
	\frac{g_h(x)}{h}
	\\
	&=
	\lim_{h\to0^+}
	\frac{\ImagPart\!\left[F(x+\imag h)\right]}{h}.
	\label{eq.CS_strip_limit}
\end{aligned}
	\end{align}
\end{theorem}

\begin{proof}
	Since \(V(x,0)=0\), the definition of the normal derivative at the
	real axis gives
	\begin{align*}
		\partial_yV(x,0)
		=
		\lim_{h\to0^+}
		\frac{V(x,h)-V(x,0)}{h}
		=
		\lim_{h\to0^+}
		\frac{V(x,h)}{h}.
	\end{align*}
	Since \(F\) is real-valued on the real axis, the Cauchy--Riemann
	equations give
	\begin{align*}
		\partial_xF(x)
		=
		\partial_xU(x,0)
		=
		\partial_yV(x,0).
	\end{align*}
	Therefore,
	\begin{align*}
		\partial_xF(x)
		=
		\lim_{h\to0^+}
		\frac{V(x,h)}{h}.
	\end{align*}
	
	For each \(h>0\), the imaginary component is reconstructed in the
	strip as follows  \citep{widder1961functions}
	\begin{align*}
		V(x,y)
		=
		\left(P_h(\cdot,y)*g_h\right)(x),
		\qquad 0<y<h.
	\end{align*}
	The strip
	Poisson kernel is \citep{widder1961functions}
	\begin{align}
	P_h(x,y)
	=
	\frac{1}{2h}
	\frac{\sin(\pi y/h)}
	{\cosh(\pi x/h)+\cos(\pi y/h)} .
\end{align}
By definition of the upper-boundary data,
	\begin{align}
		g_h(x)
		:=
		V(x,h)
		=
		\ImagPart\!\left[F(x+\imag h)\right].
	\end{align}
	Substituting this relation into
	eq. \eqref{eq.CS_strip_limit} gives
	\begin{align}
\begin{aligned}
			\partial_xF(x)
	&=
	\lim_{h\to0^+}
	\frac{g_h(x)}{h}
	\\
	&=
	\lim_{h\to0^+}
	\frac{\ImagPart\!\left[F(x+\imag h)\right]}{h},
\end{aligned}
	\end{align}
	which proves the theorem.
\end{proof}

\begin{theorem}[Real-component reconstruction from imaginary strip data]
	Let
	\begin{align*}
		F(z)=U(x,y)+\imag V(x,y)
	\end{align*}
	be holomorphic in a neighborhood of the closed strip
	\(0\leq y\leq h\). Assume that the imaginary component satisfies
	\begin{align*}
		V(x,0)=0,
		\qquad
		V(x,h)=g_h(x).
	\end{align*}
	Then the harmonic reconstruction of \(V\) is
	\begin{align*}
		V(x,y)
		=
		\left(P_h(\cdot,y)*g_h\right)(x),
	\end{align*}
	where \(P_h\) is the Poisson kernel associated with the upper
	boundary of the strip (eq. \eqref{eq.CS_strip_Poisson}) (see \cite{widder1961functions}). The corresponding real component is given,
	up to an additive constant, by
	\begin{align*}
		U(x,y)
		=
		\left(Q_h(\cdot,y)*g_h\right)(x)+C,
	\end{align*}
	where
	\begin{align*}
		Q_h(x,y)
		=
		\frac{1}{2h}
		\frac{\sinh(\pi x/h)}
		{\cosh(\pi x/h)+\cos(\pi y/h)} ,
	\end{align*}
	is the conjugate Poisson kernel for the strip. In particular, at the
	real boundary \(y=0\),
	\begin{align}
		U(x,0)
		=
		\left[
		\frac{1}{2h}
		\tanh\left(
		\frac{\pi\,\cdot}{2h}
		\right)
		*g_h
		\right](x)+C.
	\end{align}
\end{theorem}
  
\begin{proof}
	The imaginary component is reconstructed in the strip as
	\begin{align*}
		V(x,y)
		=
		\left(P_h(\cdot,y)*g_h\right)(x).
	\end{align*}
	Since \(U\) and \(V\) are harmonic conjugates, the
	Cauchy--Riemann equations give
	\begin{align*}
		\partial_xU=\partial_yV,
		\qquad
		\partial_yU=-\partial_xV.
	\end{align*}
	The strip Poisson kernel \(P_h\) and its conjugate kernel \(Q_h\)
	satisfy
	\begin{align*}
		\partial_xQ_h=\partial_yP_h,
		\qquad
		\partial_yQ_h=-\partial_xP_h.
	\end{align*}
	Therefore the corresponding real component is
	\begin{align*}
		U(x,y)
		=
		\left(Q_h(\cdot,y)*g_h\right)(x)+C,
	\end{align*}
	where \(C\) is an additive constant. At the real boundary \(y=0\),
	\begin{align*}
		\begin{aligned}
			Q_h(x,0)
			&=
			\frac{1}{2h}
			\frac{\sinh(\pi x/h)}
			{\cosh(\pi x/h)+1}
			\\
			&=
			\frac{1}{2h}
			\tanh\left(\frac{\pi x}{2h}\right).
		\end{aligned}
	\end{align*}
	Hence
	\begin{align*}
		U(x,0)
		=
		\left[
		\frac{1}{2h}
		\tanh\left(
		\frac{\pi\,\cdot}{2h}
		\right)
		*g_h
		\right](x)+C,
	\end{align*}
	which proves the theorem.
\end{proof}

\begin{remark}
	The strip formulation provides a finite-height interpretation of the
	complex-step construction. For each \(h>0\), the quantity
	\begin{align*}
		g_h(x)
		=
		\ImagPart\!\left[F(x+\imag h)\right]
	\end{align*}
	is the value of the imaginary component at the upper boundary of the
	strip, while \(V(x,0)=0\) at the lower boundary. The strip Poisson
	kernel reconstructs the imaginary component between these two
	boundaries. As the strip width tends to zero, the scaled upper-boundary
	value \(g_h(x)/h\) converges to the normal derivative
	\(\partial_yV(x,0)\), which, by the Cauchy--Riemann equations, equals
	the derivative of the real component along the real axis. Thus, the
	usual complex-step eq. \eqref{eq.Complex_Step_relation} is recovered as the shrinking-strip limit
	of the harmonic reconstruction.
\end{remark}

	\subsection{The complex step method from Cauchy reconstruction}

The complex-step construction admits a natural interpretation as a Cauchy boundary-value problem for a harmonic function. 

\begin{theorem}[Complex-step differentiation from Cauchy reconstruction]
	Let
	\begin{align*}
		F(z)=U(x,y)+\imag V(x,y),
		\qquad z=x+\imag y,
	\end{align*}
	be holomorphic in a neighborhood of the real axis and real-valued on
	the real axis. Then
	\begin{align*}
		V(x,0)=0.
	\end{align*}
	If the complementary Cauchy datum is denoted by
	\begin{align*}
		\partial_yV(x,0)=g(x),
	\end{align*}
	then
	\begin{align}
		\partial_xF(x)
		=
		\lim_{h\to0^+}
		\frac{\ImagPart[F(x+\imag h)]}{h}
		=
		g(x).
		\label{eq.CS_Cauchy_reconstruction}
	\end{align}
\end{theorem}

\begin{proof}
	From the Cauchy--Riemann equations, and since \(F\) is real-valued
	on the real axis,
	\begin{align*}
		\partial_xF(x)
		=
		\partial_xU(x,0)
		=
		\partial_yV(x,0)
		=
		g(x).
	\end{align*}
	By the definition of the normal derivative (and using  \(V(x,0)=0\)) we can write,
	\begin{align*}
		\partial_yV(x,0)
		=
		\lim_{h\to0^+}
		\frac{V(x,h)}{h}.
	\end{align*}
	Finally,
	\begin{align*}
		V(x,h)
		=
		\ImagPart[F(x+\imag h)],
	\end{align*}
	and therefore
	\begin{align*}
		\partial_xF(x)
		=
		\lim_{h\to0^+}
		\frac{\ImagPart[F(x+\imag h)]}{h}
		=
		g(x),
	\end{align*}
	which proves the theorem.
\end{proof}

Thus, the complex-step expression emerges directly from the normal derivative associated with the harmonic Cauchy problem. In this interpretation, the imaginary displacement $h$ corresponds to a displacement away from the real axis along the transverse direction.

This formulation also clarifies the distinction with reconstruction from Dirichlet data in the upper half-plane. For prescribed values on the boundary, harmonic continuation can be represented by the Poisson kernel. In the present Cauchy formulation, however, \(V(x,0)\) is identically zero and the derivative information is carried by the normal datum \(\partial_yV(x,0)=g(x)\).

Cauchy reconstruction for Laplace's equation is fundamentally different from Poisson reconstruction. In particular, the Cauchy problem requires compatible analytic data and is, for general perturbed data, ill posed. In the complex-step setting, however, the harmonic components $U$ and $V$ arise from the same holomorphic function and are constrained by the Cauchy--Riemann relations. The complex-step method can consequently be interpreted as the recovery of the normal Cauchy datum from the imaginary component evaluated arbitrarily close to the real boundary.

   \section{Discussion}

   \subsection{Reconstruction from an arbitrary measure}	
	
	The previous sections developed two closely related harmonic reconstruction problems. First, arbitrary real boundary data were reconstructed in the upper half-plane through the Poisson and conjugate Poisson integral operators. Second, the complex-step method was formulated as a harmonic reconstruction in a strip, where the imaginary component vanishes on the real axis and its value at the upper boundary encodes the complex-step perturbation. We now extend the first of these reconstruction problems from ordinary boundary functions to finite measures. When the boundary data cannot be represented by ordinary functions alone, the natural extension is obtained by replacing the boundary density \(u_0(x)\) by a positive finite measure \(\mu\).
	\begin{definition}[Stieltjes Transform]
		Let $\mu$ be a positive, finite measure on the real line. The \emph{Stieltjes transform} of $\mu$ is defined as \citep{anderson2010introduction,HirschmanWidder1955,WidderLaplaceTransform}
		\begin{align}
			S_{\mu}(z) =  \int_{-\infty}^{\infty} \frac{d\mu(t)}{t - z}, \quad z \in \mathbb{C} \setminus \mathbb{R}.
			\label{eq.Stieltjes_Transform}
		\end{align}
	\end{definition}
	
	For $z \in \mathbb{C} \setminus \mathbb{R}$, both the real and imaginary parts of $(t-z)^{-1}$ are continuous functions of $t \in \mathbb{R}$, and are absolutely integrable with respect to any finite measure~$\mu$. Thus, the Stieltjes transform is the natural measure-theoretic extension of the Cauchy reconstruction formula
	\citep{anderson2010introduction,HirschmanWidder1955,WidderLaplaceTransform}.
	
	The Stieltjes inversion transform allows one to recover the original measure~$\mu$ from its transform~$S_\mu(z)$.
	
	\begin{definition}[Stieltjes Inverse Transform]
		For any open interval $I\in[a,b]$ whose endpoints are not atoms of $\mu$ (that is, points carrying a positive, nonzero amount of measure) \citep{anderson2010introduction,HirschmanWidder1955,WidderLaplaceTransform}, the measure of the interval is recovered as
		\begin{align}
			\begin{aligned}
				\mu((a,b]) 
				&= \lim_{\epsilon \to 0^+} \frac{1}{\pi} \int_I \frac{S_\mu(x + \imag\epsilon) - S_\mu(x - \imag\epsilon)}{2 \, \imag} \, dx \\
				&= \lim_{\epsilon \to 0^+} \frac{1}{\pi}  \int_I \Im \, S_\mu(x + \imag\epsilon) \, dx .
				\label{eq.Stieltjes_Inverse_Formula}
			\end{aligned}
		\end{align}
	\end{definition}
	
	The Stieltjes inversion formula thus provides a way to reconstruct the measure~$\mu$ from its transform~$S_{\mu}(z)$. The Stieltjes inversion formula also shows that the imaginary component of the analytic continuation contains the complete information required to reconstruct the original measure. In particular,
	\begin{align}
		\ImagPart\left[S_\mu(x+\imag\epsilon)\right]
		=
		\pi\,\frac{d\mu_\epsilon}{dx},
		\qquad \epsilon>0,
	\end{align}
	where \(d\mu_\epsilon/dx\) denotes the Poisson regularization of the measure \(\mu\).
	
	This result provides the measure-theoretic extension of the harmonic reconstruction considered above. For ordinary boundary functions, the Poisson integral reconstructs the harmonic extension of the prescribed boundary data. For arbitrary finite measures, the imaginary component of the Stieltjes transform gives the corresponding Poisson regularization of the measure. Thus, the same harmonic reconstruction principle extends naturally from ordinary boundary functions to finite measures, with the Stieltjes transform providing the corresponding complex-analytic representation.
	
    \subsection{Spectral theory and resolvents}
    
    The previous section showed that the complex-step method admits a natural extension from ordinary functions to arbitrary finite measures (using the Stieltjes transform). Spectral theory provides a natural application of this result because the spectral information of an operator can itself be encoded in a Stieltjes transform through its resolvent \citep[e.g.][]{hislop2012introduction,zworski2012semiclassical}.
    
    To make this connection understandable, consider a general operator \(P\). The spectrum of \(P\), denoted by \(\sigma(P)\), is the set of values \(z\), where \(z\in\mathbb{C}\) is the complex spectral parameter, for which \((P-z)\) does not possess a bounded inverse. The simplest example is provided by an eigenvalue. If
    \begin{align}
    	Pu_n=\lambda_nu_n,
    \end{align}
    then \(\lambda_n\) is a spectral value of \(P\). Equivalently,
    \begin{align}
    	(P-\lambda_n)u_n=0,
    \end{align}
    so that \((P-\lambda_n)\) cannot be inverted in the usual sense. Eigenvalues constitute one part of the spectrum; more generally, the spectrum may also contain a continuous component. For values \(z\) outside the spectrum, one may consider the inverse
    \begin{align}
    	R(z)=(P-z)^{-1}.
    \end{align}
    This inverse is called the resolvent of \(P\).  If \(P\) is self-adjoint, its spectrum is contained in the real axis, \(\sigma(P)\subset\mathbb{R}\), so that the resolvent is well defined for every \(z\in\mathbb{C}\setminus\mathbb{R}\).
    
    Let \(P\) be a self-adjoint operator on a Hilbert space \(\mathcal{H}\), and let \(E(\lambda)\) denote its projection-valued spectral measure. Then the resolvent of \(P\) admits the representation
    \begin{align}
    	(P-z)^{-1}
    	=
    	\int_{\mathbb{R}}
    	\frac{1}{\lambda-z}\,dE(\lambda),
    	\qquad
    	z\in\mathbb{C}\setminus\mathbb{R}.
    \end{align}
    For any \(u\in\mathcal{H}\), define the associated scalar spectral measure \(\mu_u\) by
    \begin{align*}
    	\mu_u(B)
    	=
    	\langle E(B)u,u\rangle,
    \end{align*}
    for every Borel set \(B\subset\mathbb{R}\). Taking the quadratic form of the resolvent gives
    \begin{align}
    	\left\langle (P-z)^{-1}u,u\right\rangle
    	=
    	\int_{\mathbb{R}}
    	\frac{d\mu_u(\lambda)}{\lambda-z},
    	\label{eq.quadratic_resolvent}
    \end{align}
    which shows that the scalar matrix element of the resolvent is precisely the Stieltjes transform of the spectral measure \citep[e.g.][]{demuth2005determining,akhiezer2020classical,teschl2014mathematical}. 
    
    The reconstruction framework developed in the previous section applies directly to spectral theory, with the only difference that the analytic continuation is now given by the scalar resolvent rather than by an arbitrary analytic function. Evaluating the resolvent at the same imaginary perturbation employed by the complex-step method,
    \begin{align*}
    	z=x+\imag\epsilon,
    	\qquad
    	\epsilon>0,
    \end{align*}
    gives
    \begin{align}
    	\ImagPart
    	\left\langle
    	\left(P-(x+\imag\epsilon)\right)^{-1}u,u
    	\right\rangle
    	=
    	\int_{\mathbb{R}}
    	\frac{\epsilon}
    	{(\lambda-x)^2+\epsilon^2}
    	\,d\mu_u(\lambda).
    \end{align}
    
    Since
    \begin{align}
    	P_{\epsilon}(x)
    	=
    	\frac{1}{\pi}
    	\frac{\epsilon}{x^2+\epsilon^2},
    \end{align}
    is the Poisson kernel, we obtain
    \begin{align}
    	\frac{1}{\pi}
    	\ImagPart
    	\left\langle
    	\left(P-(x+\imag\epsilon)\right)^{-1}u,u
    	\right\rangle
    	=
    	\left(P_{\epsilon}*\mu_u\right)(x).
    \end{align}
    Thus, the imaginary component generated by the complex-step perturbation is precisely the Poisson regularization of the spectral measure. Taking the limit
    \begin{align}
    	\frac{1}{\pi}
    	\ImagPart
    	\left\langle
    	\left(P-(x+\imag\epsilon)\right)^{-1}u,u
    	\right\rangle
    	\xrightarrow{\epsilon\to0^+}
    	\mu_u,
    	\label{eq.spectral_measure_distribution}
    \end{align}
    recovers the original spectral measure in the distributional sense.
        
    From the perspective developed in this paper, the spectral theorem reveals that the complex-step method also extends naturally to spectral theory. The previous sections showed that the imaginary component of an analytic continuation reconstructs ordinary boundary functions through the Poisson integral and, more generally, arbitrary finite measures through the Stieltjes inversion formula. Since the resolvent is itself the Stieltjes transform of the associated spectral measure, the same complex-step reconstruction principle applies directly to self-adjoint operators. Consequently, the imaginary component of the resolvent evaluated at the complex-step perturbation reconstructs the underlying spectral measure.
        
    \subsection{Semiclassical analysis}
       
    One of the most important applications of the reconstruction principle developed in the previous sections arises in semiclassical analysis. Since the semiclassical resolvent is itself the Stieltjes transform of the associated spectral measure, the complex-step interpretation developed throughout this paper applies directly to one of the central objects of modern quantum mechanics.
    
    The central object of semiclassical analysis is the semiclassical Schr\"odinger operator
    \begin{equation}
    	P_{\hbar}=-\hbar^2\Delta+V(x),
    	\label{eq.semiclassical_Schroedinger}
    \end{equation}
    where \(\Delta\) denotes the Laplacian, \(V(x)\) is the potential energy, and \(\hbar>0\) is the semiclassical parameter. The limit \(\hbar\rightarrow0\) describes the transition between quantum and classical mechanics \citep{zworski2012semiclassical}. The corresponding semiclassical resolvent
    \begin{equation}
    	R_{\hbar}(z)
    	=
    	(P_{\hbar}-z)^{-1},
    	\label{eq.resolvendt}
    \end{equation}
    is one of the central analytical objects of semiclassical analysis.
    
    For a self-adjoint operator, the resolvent encodes the spectral properties of the system and its boundary values determine the associated spectral measure. Its asymptotic behaviour as \(h\rightarrow0\) plays a central role in the analysis of wave propagation, scattering, and the correspondence between quantum and classical theories
    \citep{zworski2012semiclassical,dyatlov2019mathematical,teschl2014mathematical}.
    
    Since \(P_h\) is self-adjoint, the discussion of the previous section applies without modification. In particular,
    \begin{equation}
    	\langle R_{\hbar}(z)u,u\rangle
    	=
    	\int_{\mathbb R}
    	\frac{d\mu_h(\lambda)}
    	{\lambda-z},
    \end{equation}
    which means that the scalar matrix elements of the semiclassical resolvent are Stieltjes transforms of the corresponding spectral measures. Evaluating the semiclassical resolvent at the same imaginary perturbation employed by the complex-step method, the term 
    \begin{align}
    	\ImagPart \left\langle
    	R_{\hbar}(x+\imag\epsilon)u,u
    	\right\rangle ,
    \end{align}
    is the Poisson regularization of the corresponding spectral measure, exactly as in the general Stieltjes inversion formula. Thus, the imaginary component generated by the complex-step perturbation is precisely the Poisson regularization of the semiclassical spectral measure.
   
   \subsection{A general reconstruction framework}
   
   The previous sections reveal a common complex-analytic structure involving
   three closely related reconstruction problems. The complex-step method
   admits two complementary harmonic formulations. From data prescribed on
   the real axis, it can be interpreted as a Cauchy problem, in which the
   imaginary component vanishes and its normal derivative provides the
   derivative of the real component. Alternatively, for a finite complex
   displacement \(h\), it can be formulated as a Dirichlet problem in a
   strip, in which the imaginary component vanishes at the lower boundary
   and its value at the upper boundary is given by the complex-step
   perturbation. In the limit as the strip width tends to zero, the
   finite-height strip data recover the normal Cauchy datum.
   
   A third reconstruction problem arises in the upper half-plane. There,
   prescribed boundary values are reconstructed through the Poisson and
   conjugate Poisson kernels. This construction extends naturally from
   ordinary boundary functions to finite measures through the Stieltjes
   transform and, consequently, to spectral measures through the
   resolvent.
   
   All these formulations are related through the harmonic-conjugate structure
   of holomorphic functions, but correspond to different choices of domain
   and boundary data. The Cauchy and strip formulations provide two
   complementary interpretations of complex-step differentiation, whereas
   the upper-half-plane formulation describes reconstruction from prescribed
   boundary values and its extension to measures and spectral measures. The reconstruction problems considered in this work are summarized in
   Table~\ref{tab:unifying_reconstruction}.
   
   \begin{table}
   	\centering
   	\caption{Different reconstruction problems considered in this work.
   		The complex-step method admits complementary Cauchy and strip
   		formulations, while boundary functions, finite measures, and spectral
   		measures are reconstructed through the upper-half-plane framework.}
   	\label{tab:unifying_reconstruction}
   	\renewcommand{\arraystretch}{2.0}
   	\begin{tabular}{l p{4.5cm} c l}
   		\hline
   		\textbf{Problem}
   		&
   		\textbf{Complex representation}
   		&
   		\textbf{Operation}
   		&
   		\textbf{Recovered quantity}
   		\\
   		\hline
   		
   		Complex step:
   		Cauchy
   		&
   		Harmonic Cauchy problem
   		&
   		\(\displaystyle
   		\frac{1}{h}
   		\ImagPart\!\left[F(x+\imag h)\right]\)
   		&
   		\(\partial_xF(x)\), as \(h\to0^+\)
   		\\
   		
   		Complex step:
   		strip
   		&
   		Strip Poisson reconstruction
   		&
   		\(\displaystyle
   		\frac{1}{h}
   		\ImagPart\!\left[F(x+\imag h)\right]\)
   		&
   		\(\partial_xF(x)\)
   		\\
   		
   		Boundary data \(u_0\)
   		&
   		Poisson/Cauchy extension
   		&
   		\(\displaystyle
   		(P_y*u_0)(x)\)
   		&
   		\(u_0\), as \(y\to0^+\)
   		\\
   		
   		Finite measure \(\mu\)
   		&
   		Stieltjes transform \(S_\mu\)
   		&
   		\(\displaystyle
   		\frac{1}{\pi}
   		\ImagPart\!\left[S_\mu(x+\imag y)\right]\)
   		&
   		\(\mu\), as \(y\to0^+\)
   		\\
   		
   		Spectral measure \(\mu_u\)
   		&
   		Scalar resolvent
   		&
   		\(\displaystyle
   		\frac{1}{\pi}
   		\ImagPart\!
   		\left\langle
   		\left(P-(x+\imag y)\right)^{-1}u,u
   		\right\rangle\)
   		&
   		\(\mu_u\), as \(y\to0^+\)
   		\\
   		\hline
   	\end{tabular}
   \end{table}
    
	\section{Conclusions}
	
    The complex-step method, introduced by \citet{Squire1998}, is commonly
    presented as a numerical differentiation technique. In this work, we
    have shown that it admits a natural interpretation within the classical
    theory of harmonic reconstruction. From this perspective, the imaginary
    perturbation arises from the analytic continuation of a holomorphic
    function, whose real and imaginary components are related through the
    Cauchy--Riemann equations and satisfy the Laplace equation.
    
    The complex-step method admits two complementary harmonic
    interpretations. First, it can be formulated as a Cauchy problem on the
    real boundary, where the imaginary component vanishes and its normal
    derivative provides the derivative information of the real component.
    The complex-step expression then follows directly from the recovery of
    this normal Cauchy datum from values of the imaginary component
    arbitrarily close to the real axis. Second, the method can be formulated
    as a Dirichlet reconstruction problem in a strip, where the imaginary
    component vanishes at the lower boundary and its value at the upper
    boundary is given by the complex-step perturbation. The strip Poisson
    and conjugate Poisson kernels provide the corresponding harmonic
    reconstruction and explicitly relate the derivative information on the
    real axis to the imaginary component within the strip.
    
    A related reconstruction problem arises in the upper half-plane. There,
    arbitrary real boundary data are reconstructed through the Poisson and
    conjugate Poisson kernels. Extending this construction from ordinary
    boundary functions to finite measures leads naturally to the Stieltjes
    transform and its inversion formula. The same measure-theoretic
    structure appears in spectral theory, where scalar matrix elements of
    the resolvent are Stieltjes transforms of the associated spectral
    measures.
    
    The results therefore reveal several related harmonic reconstruction
    structures. The Cauchy and strip formulations provide two complementary
    interpretations of the complex-step method, while the upper-half-plane
    construction connects harmonic boundary reconstruction with finite
    measures, the Stieltjes transform, and spectral measures. These
    constructions share the harmonic-conjugate structure of holomorphic
    functions, while their different domains and boundary conditions
    determine the corresponding reconstruction problem and the quantity
    that is recovered.
    
    The main contribution of this work is the identification of the
    harmonic reconstruction structure underlying the complex-step method.
    Its Cauchy formulation identifies complex-step differentiation with the
    recovery of a normal boundary datum, while its strip formulation
    provides an explicit reconstruction in terms of the Poisson and
    conjugate Poisson kernels. Together, these interpretations place the
    complex-step method within the classical theory of harmonic and
    holomorphic reconstruction, while preserving its distinction from the
    related reconstruction problems for boundary functions, finite
    measures, and spectral measures.

	\section{Acknowledgments}
	
	C.N. acknowledges financial support from the IPGP. The authors gratefully acknowledge an anonymous reviewer for very constructive comments that helped us to improve the manuscript.

	\footnotesize
	\bibliographystyle{apalike}
	\bibliography{Biblio}

\end{document}